\documentclass[a4paper]{amsart}
\usepackage{amsgen, amstext,amsbsy,amsopn, amsthm, amsfonts,amssymb,amscd,amsmath,euscript,enumerate,url,verbatim,calc,xypic,tikz}
\usepackage{enumitem}
 \usepackage{xcolor}
 \usepackage{tikz-cd}
\usepackage{amsgen, amstext, amsbsy, amsopn, amsthm, amsfonts, amssymb, amscd, amsmath, euscript, enumerate, url, verbatim, calc, xypic, tikz}
\usepackage{amsfonts}
\usepackage{amssymb} 
\usepackage{graphicx}

\usepackage{mathrsfs}  
\usepackage{hyperref}
\usepackage{MnSymbol}
\usepackage{pgfkeys}
\usepackage{mathtools}
 \usepackage{MnSymbol}
\usepackage{pgfkeys}
\usepackage{mathtools}
\def\multiset#1#2{\ensuremath{\left(\kern-.3em\left(\genfrac{}{}{0pt}{}{#1}{#2}\right)\kern-.3em\right)}}
\DeclarePairedDelimiter\ceil{\lceil}{\rceil}

\usetikzlibrary{arrows}

\usepackage{latexsym}
\usepackage{graphics}
\usepackage{color}
\usepackage{lastpage}
\usepackage{fancyhdr}
\usepackage{multirow}
\allowdisplaybreaks
\usepackage{graphicx}
\graphicspath{ {F:/IMAGES/} }

\newcommand{\IFF}{\text{if and only if}}

\newcommand{\m}{\mathfrak{m} }

\newcommand{\depth}{\operatorname{depth}}

\newcommand{\proset}{\,\mathrel{\lower 4pt\hbox{$\scriptscriptstyle/$}
		\mkern -14mu\subseteq }\,} 
		\usepackage{stackengine}
\newcommand\tsup[2][2]{%
 \def\useanchorwidth{T}%
  \ifnum#1>1%
    \stackon[-.5pt]{\tsup[\numexpr#1-1\relax]{#2}}{\scriptscriptstyle\sim}%
  \else%
    \stackon[.5pt]{#2}{\scriptscriptstyle\sim}%
  \fi%
}
\newtheorem{theorem}{Theorem}[section]
\newtheorem{corollary}[theorem]{Corollary}
\newtheorem{lemma}[theorem]{Lemma}
\newtheorem{proposition}[theorem]{Proposition}

\usepackage{amsmath}

\theoremstyle{definition}

\newtheorem{remark}[theorem]{Remark}
\newtheorem{definition}[theorem]{Definition}

\newtheorem{example}[theorem]{Example}
\title[Ratliff-Rush filtration, Hilbert coefficients and reduction number]{Ratliff-Rush filtration, Hilbert coefficients and reduction number of integrally closed ideals}

\author[Kumari Saloni and Anoot Kumar Yadav]{Kumari Saloni and Anoot Kumar Yadav}
\thanks{ The second author is supported by  a UGC fellowship, Govt. of India.}
\subjclass[2020]{13H10, 13D40, 13A30}
\keywords{Cohen-Macaulay local rings, reduction number, Ratliff-Rush filtration, Hilbert coefficients}
\address{Department of Mathematics, Indian Institute of Technology Patna, Bihta, Patna 801106, India}
\email{ksaloni@iitp.ac.in}
\address{Department of Mathematics, Indian Institute of Technology Patna, Bihta, Patna 801106, India} 
\email{anoot\_2021ma06@iitp.ac.in,\; vicky.anoot@gmail.com}
\begin{document}
\begin{abstract}
    Let $(R,\m)$ be a Cohen-Macaulay local ring of dimension $d\geq 3$ and $I$ an integrally closed $\m$-primary ideal. We establish  bounds for the third Hilbert coefficient $e_3(I)$  in terms of the lower Hilbert coefficients $e_i(I),~0\leq i\leq 2$ and  the  reduction number of $I$. When $d=3$, the boundary cases of these bounds characterize certain 
    properties of the Ratliff-Rush filtration of $I$. These properties, though weaker than $\depth G(I)\geq 1$, guarantees that  Rossi's bound for reduction number $r_J(I)$ holds in dimension three. In that context, we prove that if $\depth G(I)\geq d-3$, then $r_J(I)\leq e_1(I)-e_0(I)+\ell(R/I)+1+e_2(I)(e_2(I)-e_1(I)+e_0(I)-\ell(R/I))-e_3(I).$  We also discuss the signature of the fourth Hilbert coefficient $e_4(I).$
\end{abstract}
\maketitle
\section{Introduction}
A celebrated result of Rossi \cite{r} gives a bound for the reduction number of an $\m$-primary ideal $I$ in a Cohen-Macaulay local ring  $(R,\m)$ of dimension two. Rossi's bound is linear in terms of the multiplicity and the first Hilbert coefficient of $I$. It is an open problem to establish the same bound in higher dimensions. This paper is inspired by our interests in  extending Rossi's proof in dimension three.  We have two important observations.   Firstly, higher Hilbert coefficients  of the ideal may appear in the bound, see Theorem \ref{improvment ms result}.  Next, it is desirable that the Ratliff-Rush filtration of the ideal behaves well in some sense, see Remark \ref{re4}.


Throughout the paper, let $(R,\m)$ be a Cohen-Macaulay  local ring of dimension $d\geq 1$ with infinite residue field and $I$ an $\m$-primary ideal. 
 An $I$-admissible filtration  $\mathcal I=\{I_n\}_{n\in \mathbb Z}$ is a sequence of ideals  such that $(i)$ $I_{n+1}\subseteq I_n$,
   $(ii)$  $I_m I_n\subseteq I_{m+n}$ and $(iii)~ I^n\subseteq I_n\subseteq I^{n-k}$ for some $k\in \mathbb N$. A {\it reduction} \index{reduction} of $\mathcal I$ is an ideal $J\subseteq I_1$ such that $JI_n=I_{n+1}$ for  $n\gg 0$ and it is called a {\it minimal reduction} if it is  minimal with respect to containment among all reductions. If the residue field $R/\m$ is infinite, then  a minimal reduction of $\mathcal I$ exists and is generated by $d$ elements. Minimal reductions are important in the study of Hilbert functions and blow-up algebras. For a minimal reduction $J$ of $\mathcal I$, we define 
    \[r_J(\mathcal I)= \sup \{n\in \mathbb Z \mid I_n\not= JI_{n-1}\} \mbox{ and } 
  r(\mathcal I)=\min\{r_J(\mathcal I)~|~ J \mbox{ is a minimal reduction of } \mathcal I\}.\]
 The above numbers are known as {\it{reduction number of $\mathcal I$ with respect to $J$}} and  {\it{reduction number of $\mathcal I$}}  respectively. When $\mathcal I$ is the $I$-adic filtration $\{I^n\}_{n\in\mathbb{Z}}$,  we write $r_J(I)$ and $r(I)$ in place of $r_J(\mathcal I)$ and $r(\mathcal I)$ respectively. 
 
One of the main difficulties in extending Rossi's bound in higher dimensions is 
that  $r_J(I)$ does not behave well with respect to superficial elements. This fact is closely related to the Ratliff-Rush filtration of $I$. We recall the definition more generally for an $I$-admissible filtration $\mathcal I$. The Ratliff-Rush filtration  of $\mathcal I$ is the filtration $\{\widetilde{I_n}=\mathop\bigcup\limits_{t\geq 0}(I_{n+t}:I^t)\}_{n\in\mathbb{Z}}$. 
Let   $x\in I$ be a superficial element for $I$ and $R^\prime=R/(x)$. We say that the {\it{Ratliff-Rush filtration of $I$ behaves well modulo  $x$}} if $\widetilde{I^n}R^\prime = \widetilde{I^nR^\prime}$
 for all $n\geq 1.$  
In Remark \ref{re4}, we prove that in a three dimensional Cohen-Macaulay local ring, if the Ratliff-Rush filtration of $I$ behaves well modulo a superficial element $x\in I$, then Rossi's bound holds for $r_J(I)$. In view of this, it is natural to look for conditions, preferably computable ones, when the Ratliff-Rush filtration of $I$ behaves well.  
Hilbert coefficients provide some control over this property as we will see in this paper. In dimension three, a sufficient condition, due to Puthenpurakal \cite{Put2}, is  the vanishing of the Hilbert coefficients $e_2(I)$ and $e_3(I)$ (defined below). However, if $I$ is integrally closed, then $e_2(I)=0$ implies that $G(I)$ is Cohen-Macaulay which is a stronger condition.

The Hilbert coefficients of $\mathcal I$ are the unique integers $e_i(\mathcal I)$, $0\leq i\leq d,$ such that the function $H_{\mathcal I}(n):=\ell(R/I_n)$ coincides with the following polynomial for $n\gg 0:$  $$P_{\mathcal I}(x) =e_0(\mathcal I){x+d-1\choose d}-e_1(\mathcal{ I}){x+d-2\choose d-1}+\cdots +(-1)^de_d(\mathcal I).$$ 
Here $\ell(*)$ denotes the length function. The function $H_{\mathcal I}(n)$ and the polynomial $P_{\mathcal I}(x)$ are known as the Hilbert-Samuel function and the Hilbert-Samuel polynomial of $\mathcal I$ respectively. We refer to \cite{rv} for details.  For $\mathcal I=\{I^n\}_{n\in\mathbb{Z}}$, we write $e_i(I)$ instead of $e_i(\mathcal I)$. 
The Hilbert coefficients contain deep information about the ideal $I$, the ring $R$ and the blow-up algebras associated to $I$. 

There has been a lot of research focused at the first three coefficients, i.e., $e_i(I)$ for $0\leq i\leq 2.$ 
The behaviour of these coefficients change as soon as we encounter $e_3(I).$
For instance, Northcott \cite{north} proved that $e_1(I)\geq 0$ and Narita \cite{na} showed that $e_2(I)\geq 0$. In the same paper,  Narita showed that $e_3(I)$ could be negative.   However, $e_3(I)\geq 0$ if 
$\depth G(I)\geq d-1$ due to Marley \cite{tm} or if $I$ is a normal ideal due to Itoh \cite{it} and Huckaba-Huneke \cite{sc}. 
Later, Corso-Polini-Rossi \cite{cpr} showed that $e_3(I)\geq 0$ if $I^q$ is integrally closed for a sufficiently large $q$. When $I$ is integrally closed, Itoh \cite{itoh} proved that  $e_3(I)=0$ if $e_2(I)\leq 2.$ 
In \cite{Put2}, Puthenpurakal proved that, for  an $\m$-primary ideal  $I$, $e_3(I)\leq0$ if $e_2(I)=0$ or if $r_J(I)\leq 2$.
  Further, Mafi-Naderi  \cite {mf2}, proved that $e_3(I)\leq 0$, if $e_2(I)\leq 1$. When $I$ is integrally closed, we recover some of these results as easy consequences of our main theorems. We obtain that $e_3(I)\leq 0$ if $e_2(I)\leq 2$ and $e_3(I)\leq 2$ if $e_2(I)= 3$. Discussions on the signature and vanishing of Hilbert coefficients   can be found in  \cite{ct}, \cite{lt}, \cite{ggo}, \cite{go}, \cite{lin}, \cite{mc}, \cite{mbv}, \cite{ak} etc. 
  

Suppose $\depth G(I)\geq d-1$. Then Marley \cite{tm} proved that $e_3(I)\geq e_2(I)-e_1(I)+e_0(I)-\ell (R/I)$. The following theorem of this paper is a generalization of Marley's result. Marley's hypothesis is stronger as evident by Example \ref{ex1.2}.
\begin{theorem}\label{theorem-t1}
Let $(R,\m)$ be a Cohen-Macaulay local ring of dimension $d\geq 3$ and $I$ an $\m$-primary ideal. Suppose the Ratliff-Rush filtration of $I$ behaves well modulo a superficial sequence $x_1,\ldots,x_{d-2}\in I$. Then $e_3(I)\geq e_2(I)-e_1(I)+e_0(I)-\ell(R/I).$ 
\end{theorem}
The above theorem provides a necessary condition for the Ratliff-Rush filtration of $I$ to behave well modulo a superficial element. Suppose $I$ is an integrally closed ideal.  In Corollary \ref{final-conclusion}, we show that 
$e_3(I)\leq e_2(I)-e_1(I)+e_0(I)-\ell(R/I)$
if one of the following conditions holds:   (i) $r_J(I)\leq 3$;  (ii) $e_2(I)\leq 3$; (iii) $e_1(I)-e_0(I)+\ell(R/I)\leq 2$ or (iv) $e_2(I)-e_1(I)+e_0(I)-\ell(R/I)\leq 1.$
In fact when $d=3$, we prove that in all these cases, the equality $e_3(I)= e_2(I)-e_1(I)+e_0(I)-\ell(R/I)$ holds if and only if  the Ratliff-Rush filtration of $I$ behaves well modulo a superficial element. These are consequences of the following  theorem. 



\begin{theorem}\label{all-theorems}
Let $(R,\m)$ be a Cohen-Macaulay local ring of dimension $d\geq 3$, $I$ an integrally closed $\m$-primary ideal and $J$ a minimal reduction of $I$. Then 
\begin{equation}\label{01}
    e_3(I)\leq  \frac{(r_J(I)-1)}{2}\big(e_2(I)-e_1(I)+e_0(I)-\ell ({R}/{I})\big);
\end{equation}
\begin{align}\label{02}
e_3(I)\leq \frac{\big(e_1(I)-e_0(I)+\ell(R/I)\big)}{2} \big(e_2(I)-e_1(I)+e_0(I)-\ell ({R}/{I})\big) \text{ and }\end{align}
\begin{align}\label{03} 
e_3(I)\leq \frac{\big(e_2(I)-1\big)}{2} \big(e_2(I)-e_1(I)+e_0(I)-\ell ({R}/{I})\big).\end{align}
Further, suppose $d=3$ and equality holds in any one of \eqref{01}, \eqref{02} or \eqref{03}. Then the Ratliff-Rush filtration of $I$ behaves well modulo a superficial element. 
Conversely, suppose the Ratliff-Rush filtration of $I$ behaves well modulo a superficial sequence $x_1,\ldots,x_{d-2}\in I$. Then  (i) equality holds in \eqref{01} provided $r_J(I)\leq 3$; (ii) equality holds in \eqref{02} provided $e_1(I)-e_0(I)+\ell(R/I)\leq 2$ and (iii) equality holds in \eqref{03} provided $e_2(I)\leq 3$.

\end{theorem}  
 The above inequalities may not hold if $I$ is not integrally closed as shown by Example \ref{example 4.7}. 
As a corollary, we obtain that $$e_3(I)\leq (e_2(I)-1)^2/2$$ for an integrally closed $\m$-primary ideal $I$. Theorem \ref{all-theorems},  provides sufficient and necessary conditions for the Ratliff-Rush filtration behaving well modulo superficial element in dimension three as demonstrated by Examples \ref{eg4} and \ref{eg-2-sec-6}. In addition, it has some interesting applications. In an attempt to extend Rossi's result in dimension three, we prove the following 
result for integrally closed ideals which gives an improvement of the 
earlier known bounds. 
\begin{theorem}\label{theorem-t3}
Let $(R,\m)$ be a Cohen-Macaulay local ring of dimension $d\geq 3$, $I$ an integrally closed $\m$-primary ideal and $J\subseteq I$ a minimal reduction of $I$. Suppose $\depth G(I)\geq d-3$. Then 
\begin{align*}
    r_J(I)\leq e_1(I)-e_0(I)+\ell (R/I)+1+e_2(I)(e_2(I)-e_1(I)+e_0(I)-\ell (R/I))-e_3(I).
\end{align*}
\end{theorem}
For  a minimal reduction $J$ of $\mathcal I$, we set $$\widetilde{r}_J(\mathcal I):=\sup\{n\in \mathbb Z|\widetilde{I_{n}}\not=J\widetilde{I_{n-1}}\}.$$ 
Suppose $d=2$. Then 
it is known that 
$\widetilde{r}_J(\mathcal{I})\leq r_J(\mathcal{I})$, see \cite{ms} and \cite{rs}. 
Assume that $I_1$ is integrally closed. In Lemma \ref{4.6}, we show that 
\begin{equation}\label{bd-1}
\widetilde{r}_J(\mathcal{I})\leq e_2(\mathcal{I})-e_1(\mathcal{I})+e_0(\mathcal{I})-\ell(R/I_1)+2.
\end{equation}

In \cite{e}, Elias proved that $\depth G(I^n)$ is constant for $n\gg0$. Also, $\depth G(I^n)\geq 1$ for $n\gg 0$ \IFF\; $I$ has a regular element. In \cite{Put1}, Puthenpurakal gave a necessary and sufficient condition for $\depth G(I^n)\geq 2.$ Suppose $d\geq 4$. In Theorem \ref{thm-e4}, we prove that (i) if $e_4(I)<0$, then $\depth G(I^n)\leq 2$ for $n\gg0$ and (ii)  if $I^q$ is integrally closed for a sufficiently large $q$ and $r_J(I)\leq 3$, then $e_4(I)\leq 0.$ The later was earlier proved for an asymptotically normal ideal \cite[Theorem 2.5]{mf2}. In Proposition \ref{cpr recover}, we partially recover \cite[Corollary 4.5]{cpr} on the vanishing of $e_3(I)$  and further prove the bound 
\begin{equation}\label{bd-e3}
 e_3(I)\leq \ell(I^3/JI^2)+\ell((J\cap I^3)/JI^2)\end{equation}
 for an integrally closed ideal $I$  with $r_J(I)\leq 3$ in Theorem \ref{2.11}. 

This paper is organised in five sections beginning with a brief mention of the notations in Section \ref{notation}. In Section \ref{results-1}, we prove Theorem \ref{theorem-t1} and the bound given in \eqref{bd-1}. Theorem \ref{all-theorems} is proved in Section \ref{results-2}. The last section of the paper presents Theorem \ref{theorem-t3}, the results on the signature of $e_4(I)$ and the bound on $e_3(I)$ given in \eqref{bd-e3}.  Our results are supported by examples.

\section{Notation}\label{notation}
{\bf{I.}} It is well known that the Hilbert series of $\mathcal I$, defined as $H(\mathcal I,t)=\mathop\sum\limits_{n\geq 0} \ell(I_n/I_{n+1})t^n$, is a rational function, i.e., there exists a unique rational polynomial $h_{\mathcal I}(t)\in\mathbb{Q}[t]$ with $h_{\mathcal I}(1)\neq 0$ such that $$H(\mathcal I,t)=\frac{h_{\mathcal I}(t)}{(1-t)^d}.$$ For every $i\geq 0$, we define $e_i(\mathcal I)=\frac{h_{\mathcal I}^{(i)}(1)}{i!}$, where $h_{\mathcal I}^{(i)}(1)$ denotes the $i$-th formal derivative of the polynomial $h_{\mathcal I}(t)$ at $t=1$. The integers $e_i(\mathcal I)$ are called the Hilbert coefficients of  $\mathcal {I}$ and for $0\leq i\leq d$, these are same as defined earlier in the Introduction. We refer to \cite{gr} for details.
We write $\widetilde{e_i}(I)$ for the coefficients $e_i(\mathcal{I})$ when $\mathcal{I}=\{\widetilde{I^n}\}_{n\geq 0}.$

{\bf{II.}} 
The postulation number $\eta(I)$ of $I$ is defined to be $\eta(I):=\min\{n\in \mathbb{Z}~|~ P_I(m)=H_I(m)$ for $ m\geq n\}$. 

{\bf{III.}} For a graded module $M=\mathop\bigoplus\limits_{n\geq 0}M_n$, $M_n$ denote its $n$-th graded piece and $M(-j)$ denote the graded module $M$ shifted to the left by degree $j.$ Similarly for a complex $C_{.}$, we write $C_{.}(-j)$ for the complex $C_{.}$ shifted to the left by degree $j.$

{\bf{IV.}} The Rees ring of  $\mathcal{I}$ is  $\mathcal{R}(\mathcal{I})=\mathop\bigoplus\limits_{n\geq 0} I_nt^n \subseteq R[t]$ where $t$ is an indeterminate and the associated graded ring of $\mathcal{I}$ is $G(\mathcal{I})=\mathop\bigoplus\limits_{n\geq0}I_n/I_{n+1}$. We write  $R(I)$ and $G(I)$ for the Rees ring and the associated graded ring respectively when $\mathcal{I}=\{I^n\}_{n\geq 0}$. Set $\mathcal{R}_+=\mathop\bigoplus\limits_{n\geq 1}I_nt^n.$

{\bf{V.}} 
Suppose $x_1,\ldots,x_s$ is a sequence in $R$. 
We write
$\underline{x}$ and $(\underline{x})$ for the sequence 
$x_1,\ldots,x_s$ and the ideal $(x_1,\ldots,x_s)\subseteq R$
respectively. We use the notation $I^\prime$ and $R^\prime$ for going modulo superficial element(s). 

{\bf{VI.}} For a numerical function $f:\mathbb  Z\longrightarrow  \mathbb Z$, we put $\triangle f(n) =f(n+1)-f(n)$ and recursively we can define $\triangle^i f(n):=\triangle(\triangle^{i-1} f(n))$ for all $i\geq 1$. 

 {\bf{VII.}} We write $\mathcal{F}$ for the filtration $\{\widetilde{I^n}\}_{n\geq 0}$ and for an $I$ admissible filtration $\mathcal{I}$
with a minimal reduction $J$, $v_n(\mathcal{I})=\ell(I_{n+1}/JI_n).$ 
For future reference, we recall the following inequality from the proof of \cite[Theorem 1.3]{r}:
\begin{equation}\label{Rossi's result}
r_J(I)\leq \mathop\sum\limits_{n\geq 0}v_n(\mathcal{F})-e_0(I)+\ell (R/I)+1.
\end{equation}

\section{Ratliff-Rush filtration and $e_3(I)$}\label{results-1}
Suppose $(R,\m)$ is a Cohen-Macaulay local ring of dimension $d\geq 3$. Marley \cite[Corollary 2.9]{tm}, proved that if $\depth G(I)\geq d-1,$ then $e_3(I)\geq e_2(I)-e_1(I)+e_0(I)-\ell(R/I)$ and Rossi \cite{r} proved that if $\depth G(I)\geq d-2,$ then $r_J(I)\leq e_1(I)-e_0(I)+\ell(R/I)+1$. As evident by the next example, both the above bounds hold even if $\depth G(I)=0$. In this section, we discuss   
that for both the results, the depth conditions on $G(I)$ can be replaced by a weaker condition on the behavior of the Ratliff-Rush filtration of $I$.

\begin{example}\cite[Example 3.7]{mf2}\label{eg1}
    Let $R= k[x,y,z]_{(x,y,z)}$ and $I=(x^4, y^4, z^4, x^3y, xy^3, y^3z, yz^3).$ Then $\depth G(I)=0$, the Hilbert series of $I$ is
    $$    H(I,t)=\frac{30+12t+22t^2+8t^3-2t^4-12t^5+6t^6}{(1-t)^3}
    $$
    and the Hilbert polynomial is
    $$
    P_I(n)=64{{n+2}\choose 3}-48{{n+1}\choose 2}+4{n \choose 1}.
    $$
    Therefore, $e_2(I)-e_1(I)+e_0(I)-\ell (R/I)=-10$ and $e_3(I)=0.$ We also note that $J=(x^4, y^4, z^4)$ is a minimal reduction of $I$ and $r_J(I)=4\leq 14=e_1(I)-e_0(I)+\ell(R/I)$. 
\end{example}
Recall that the Ratliff-Rush closure $\widetilde{I}$ of an ideal $I$ is the stable value of the following chain of ideals: 
$$I\subseteq (I^2:I)\subseteq (I^3:I^2)\subseteq\ldots\subseteq (I^{n+1}:I^n)\subseteq\ldots$$
It is of interest to understand the interaction between the Ratliff-Rush closure of powers of $I$ and a superficial element $x_1\in I$. 
In general, $\widetilde{I^n}R^\prime\subseteq \widetilde{I^nR^\prime}$
for all $n\geq 0$ and equality holds for $n\gg 0$ where $R^\prime=R/(x_1).$
\begin{definition}
Let $x_1,\ldots,x_s\in I$ be a superficial sequence. Then, we say that 
\begin{enumerate}
    \item the Ratliff-Rush filtration of $I$ behaves well modulo $x_1$ if    $\widetilde{I^n}R^\prime= \widetilde{I^nR^\prime}$
     for all $n\geq 0$ and 
    \item the Ratliff-Rush filtration of $I$ behaves well modulo $x_1,\ldots,x_s$ if $I/(x_1,\ldots,x_{i-1})$ behaves well modulo the image of $x_i$ in $R/(x_1,\ldots,x_{i-1})$ for all $1 \leq i \leq s$, where $(x_1,\ldots,x_{i-1})=(0)$ for $i=1$.
\end{enumerate}
 \end{definition}
It is easy to see that the above definition is equivalent to the one given in \cite[Definition 4.4]{Put2}.  If the Ratliff-Rush filtration of $I$ behaves well modulo one superficial sequence, then it does so for any superficial sequence.  It follows from \cite[Corollary 4.7]{Put2}.
 At this point, we recall some relevant discussions from literature, mainly \cite{Put1} and \cite{Put2}. As in \cite{Put1}, we define $L^I(R)=\mathop\bigoplus\limits_{n\geq 0} R/I^{n+1}.$ The exact sequence 
$$0\to R(I)\to R[t]\to L^I(R)(-1)\to 0$$ defines an $R(I)$-module structure on $L^I(R).$ Note that $L^I(R)$ is not a finitely generated $R(I)$-module. However, the local cohomology modules $H_{\mathcal{M}}^i (L^I(R))$, with support in $\mathcal{M}=m\bigoplus R(I)_+$,
are $*$-Artinian for $0\leq i\leq 2$, see \cite[section 4]{Put1}. For a superficial element $x\in I$, we have the following exact sequence of $R$-modules for each $n\geq 1$:
\begin{equation}
0\longrightarrow\frac{I^{n+1}:x}{I^n}\longrightarrow\frac{R}{I^n}\stackrel{\psi_n^x}{\longrightarrow}\frac{R}{I^{n+1}}\longrightarrow\frac{{R}^\prime}{I^{n+1}{{R}^\prime}}\longrightarrow 0\nonumber,
\end{equation}
where $R^\prime=R/(x)$ and $\psi_n^x(r+I^n)=xr+I^{n+1}$ for all $r\in R$. We write $I^\prime=IR^\prime$. This gives an exact sequence of $R(I)$-modules:
\begin{align}\label{exact-seq-LI-2}
0\longrightarrow B^I(x,R)\longrightarrow L^I(R)(-1){\stackrel{\psi^x}{\longrightarrow}}L^I(R)\stackrel{\rho}{\longrightarrow}{L^I}^\prime({R}^\prime)\longrightarrow 0 
\end{align}
where $\psi^x$ is multiplication by $x$ and 
$$
B^I(x,R)=\bigoplus_{n\geq 0}\frac{I^{n+1}:x}{I^n}.$$
Note that $I^{n+1}:x=I^n$ for all $n\gg 0$. So, $H_\mathcal{M}^0(B^I(x,R))=B^I(x,R)$. The sequence in \eqref{exact-seq-LI-2} induces a long exact sequence of local cohomology modules, see \cite[Section 6]{Put1}:
\begin{eqnarray}\label{sec_fun}
    0\longrightarrow B^I(x,R)&\longrightarrow 
    &H_\mathcal{M}^0(L^I(R))(-1){\longrightarrow }H_\mathcal{M}^0(L^I(R)){\longrightarrow }H_\mathcal{M}^0({L^I}^\prime({R}^\prime))\\
    &\longrightarrow & H_\mathcal{M}^1(L^I(R))(-1){\longrightarrow}H_\mathcal{M}^1(L^I(R)){\longrightarrow}H_\mathcal{M}^1({L^I}^\prime({R}^\prime)) \longrightarrow\ldots\nonumber.
\end{eqnarray}
By \cite[Proposition 4.7]{Put1}, $H_{\mathcal{M}}^0(L^I(R))=\mathop\bigoplus\limits_{n\geq 0}\frac{\widetilde{I^{n+1}}}{I^{n+1}}.$ For convenience, we write $b_I:=\ell(B^I(x,R))$ and $s_I:=\ell(H^0({L^I}^\prime({R}^\prime)).$ Note that $b_I-s_I\leq 0.$
\begin{remark}\label{remark-R}
\begin{enumerate}
\item\label{remark-RR}
 By \cite[Theorem 2.3]{Put2},  Ratliff-Rush filtration of $I$ behaves well modulo $x$ if and only if $H_{\mathcal{M}}^1(L^I(R))=0.$ Equivalently, $b_I-s_I=0.$ To see this, consider the length of modules in \eqref{sec_fun}. We have 
\begin{equation}
\ell(B^I(x,R))-\ell(H_{\mathcal{M}}^0(L^I(R)))+\ell(H_{\mathcal{M}}^0(L^I(R)))-\ell(H_{\mathcal{M}}^0({L^I}^\prime({R}^\prime)))+\ell(Im)=0\nonumber
\end{equation}
where $Im:=Image\big(H_{\mathcal{M}}^0({L^I}^\prime({R}^\prime))\longrightarrow H_{\mathcal{M}}^1(L^I(R)(-1))\big)$. Therefore, $b_I-s_I=0$ if and only if $Im=0$ if and only if 
$\ker \big(H_{\mathcal{M}}^1(L^I(R)(-1))\to H_{\mathcal{M}}^1(L^I(R))\big)=0$ which is equivalent to $H_{\mathcal{M}}^1(L^I(R))=0$ by   \cite[Lemma 1.10]{Put1} and \cite[Proposition 4.4]{Put1}. 

\item \label{remark-RR-2}
Suppose $\depth G(I)\geq 2$. Then $I^{n+1}:x=I^n$ for all $n\geq 0$ which gives $b_I=0$. Further, $\depth G(I^\prime)\geq 1$ implies $\widetilde{{I^\prime}^{n+1}}={I^\prime}^{n+1}$ for all $n\geq 0$. Therefore, $H_{\mathcal{M}}^0(L^{I^\prime}(R^\prime))=\mathop\bigoplus\limits_{n\geq 0}\frac{\widetilde{{I^\prime}^{n+1}}}{{I^\prime}^{n+1}}=0$ which gives $s_I=0$. By part \eqref{remark-RR}, Ratliff-Rush filtration of $I$ behaves well modulo a superficial element. Converse may not be true as evident by Example \ref{ex1.2}.
\item Suppose $I$ is a normal $\m$-primary ideal. Then $\frac{\widetilde{I^n}+(x)}{(x)}= \frac{I^n+(x)}{(x)}=\widetilde{\big(\frac{I^n+(x)}{(x)}\big)}$ for all $n\geq 1$ if and only if $\depth G(I^\prime)\geq 1.$ Therefore,  for a normal $\m-$primary ideal $I$, Ratliff-Rush filtration of $I$ behaves well modulo $x$ if and only if $\depth G(I)\geq 2.$ 
\end{enumerate}
\end{remark}
\begin{example}\cite[Example 3.8]{cpr}\label{ex1.2}
Let $R=Q[[x,y,z]]$ and $I=(x^2-y^2,y^2-z^2,xy,yz,xz).$ The Hilbert Series is $$
H(I,t)=\frac{5+6t^2-4t^3+t^4}{(1-t)^3}.
$$
Then $e_2(I)=e_3(I)=0$. By \cite[Theorem 6.2]{Put2}, Ratliff-Rush filtration of  $I$ behaves well modulo a superficial element. However, $\depth G(I)=0$ as $x^2\in(I^2:I)\subseteq\widetilde{I}$ but $x^2\notin I$.  Again in this case, we have $0=e_3(I)\geq e_2(I)-e_1(I)+e_0(I)-\ell(R/I)=-1.$
\end{example}

  Now we prove the following result which can be seen as a generalization of Marley's result \cite[Corollary 2.9]{tm}. 
\begin{theorem}\label{theorem-marley-gen}
Let $(R,\m)$ be a Cohen-Macaulay local ring of dimension $d \geq 3$ and $I$ an $\m$-primary ideal. Suppose the Ratliff-Rush filtration of $I$ behaves well modulo a superficial sequence $x_1,x_2,\ldots,x_{d-2}\in I$.  Then $e_3(I)\geq e_2(I)-e_1(I)+e_0(I)-\ell(R/I).$ 
\end{theorem}  
We need the following lemma. 
\begin{lemma}\label{lemma-value-of-e3}
Let $(R,\m)$ be a Cohen-Macaulay local ring of dimension three, $I$ an $\m$-primary ideal and $x\in I$ a superficial element. 
Then $$e_3(I)={\widetilde{e_3}({I}^\prime)}+b_I-s_I.$$
\end{lemma}
\begin{proof}
By \cite[Proposition 1.2]{rv}, $e_3(I)=e_3({I}^\prime)+b_I$. Since $\dim {R}^\prime=2$, we have  $e_3({I}^\prime)=\widetilde{e_3}({I}^\prime)-s_I$ 
by using \cite[1.5(b)]{Put2}. 
\end{proof}
\begin{proof}[Proof of Theorem \ref{theorem-marley-gen}]
 Suppose $d=3$. We write $R^\prime=R/(x_1)$ and $I^\prime=I/(x_1).$
 We have $e_i(I)=e_i({I}^\prime)=\widetilde{e_i}({I}^\prime)$ for $0\leq i\leq 2$. Using \cite[Corollary 4.13]{hm}, we have


\begin{align}
e_2(I)-e_1(I)+e_0(I)-\ell(R/I)&=e_2({I}^\prime)-e_1({I}^\prime)+e_0({I}^\prime)-\ell({R}^\prime/{I}^\prime)\nonumber\\
\begin{split}
    &= \widetilde{e_2}({I}^\prime)-\widetilde{e_1}({I}^\prime)+\widetilde{e_0}({I}^\prime)-\ell({R}^\prime/{I}^\prime)\label{the-equality-for-f-1}\\
   &=\sum_{n\geq 1}n\ell (\widetilde{{{I}^\prime}^{n+1}}/J^\prime\widetilde{{{I}^\prime}^n})-\sum_{n\geq 0}\ell(\widetilde{{{I}^\prime}^{n+1}}/J^\prime\widetilde{{{I}^\prime}^{n}})+\ell(R^\prime/J^\prime)-\ell({R}^\prime/{I}^\prime)\\ 
    &= \sum_{n\geq 1}n\ell (\widetilde{{{I}^\prime}^{n+1}}/J^\prime\widetilde{{{I}^\prime}^n})-\sum_{n\geq 1}\ell(\widetilde{{{I}^\prime}^{n+1}}/J^\prime\widetilde{{{I}^\prime}^{n}})-\ell(\widetilde{{I}^\prime}/J^\prime)+\ell(I^\prime/J^\prime)\\  
    &\leq  \sum_{n\geq 1}n\ell (\widetilde{{{I}^\prime}^{n+1}}/J^\prime\widetilde{{{I}^\prime}^n})-\sum_{n\geq 1}\ell(\widetilde{{{I}^\prime}^{n+1}}/J^\prime\widetilde{{{I}^\prime}^{n}})-\ell({{I}^\prime}/J^\prime)+\ell(I^\prime/J^\prime)\\
&=\sum_{n\geq 2}(n-1)\ell (\widetilde{{{I}^\prime}^{n+1}}/J^\prime\widetilde{{{I}^\prime}^n})
\end{split}
\end{align}
On the other hand, by \cite[Theorem 2.5]{rv}, we have \begin{align}\label{formula-for-e3tilde}
\widetilde{e_3}({I}^\prime)= \sum_{n\geq 2}{n \choose 2} \ell (\widetilde{{{I}^\prime}^{n+1}}/J^\prime\widetilde{{{I}^\prime}^n}).
\end{align}
Now suppose the Ratliff-Rush filtration of $I$ behaves well modulo $x_1$. Then by Lemma \ref{lemma-value-of-e3} and part \eqref{remark-RR} of Remark \ref{remark-R},  $e_3(I)={\widetilde{e_3}({I}^\prime)}=\sum_{n\geq 2}{n \choose 2} \ell (\widetilde{{{I}^\prime}^{n+1}}/J^\prime\widetilde{{{I}^\prime}^n})\geq \sum_{n\geq 2}(n-1)\ell (\widetilde{{{I}^\prime}^{n+1}}/J^\prime\widetilde{{{I}^\prime}^n})\geq e_2(I)-e_1(I)+e_0(I)-\ell(R/I).$\\
Now suppose $d\geq 4$ and the Ratliff-Rush filtration of $I$ behaves well modulo $x_1,\ldots,x_{d-2}$. Set $I^\prime=I/(x_1,\ldots,x_{d-3})$ and $R^\prime=R/(x_1,\ldots,x_{d-3})$. Then the Ratliff-Rush filtration of $I^\prime$ behaves well modulo $x_{d-2}$. By induction hypothesis,
$$
e_3(I^\prime)\geq e_2(I^\prime)-e_1(I^\prime)+e_0(I^\prime)-\ell(R^\prime/I^\prime).
$$
Since $e_i(I)=e_i(I^\prime)$ for $0 \leq i \leq 3$. Hence, $e_3(I)\geq e_2(I)-e_1(I)+e_0(I)-\ell(R/I).$
\end{proof}
As an easy consequence, we get Marley's result. 
\begin{corollary}\label{cor-mar-gen}
    Suppose $R$ is a Cohen-Macaulay local ring of dimension $d\geq 3$ and $I$ an $\m$-primary ideal of $R$. Suppose $\depth G(I)\geq d-1.$ Then $e_3(I)\geq e_2(I)-e_1(I)+e_0(I)-\ell(R/I)$.
\end{corollary}
\begin{proof}
Suppose $d=3$. Then $\depth G(I)\geq 2$ which implies Ratliff-Rush filtration of $I$ behaves well modulo a superficial element by part \eqref{remark-RR-2} of Remark \ref{remark-R}. The conclusion follows from  Theorem \ref{theorem-marley-gen}. 

Suppose $d\geq 4.$ Let $x\in I$ be a superficial element. Then $\dim R^\prime=d-1$, $\depth G(I^\prime)\geq d-2$ and $e_i(I^\prime)=e_i(I)$ for $0\leq i\leq 3.$ By induction hypothesis, $e_3(I)=e_3(I^\prime)\geq e_2(I^\prime)-e_1(I^\prime)+e_0(I^\prime)-\ell (R^\prime/I^\prime)=e_2(I)-e_1(I)+e_0(I)-\ell (R/I).$
\end{proof}
Proposition \ref{converse-of-theorem} is about the boundary case of  Theorem \ref{theorem-marley-gen} under some restrictions. For that, we need the next lemma which is interesting in itself. It is largely unknown if Rossi's bound can be generalized for reduction number of a filtration $\mathcal{I}$ in dimension two, i.e., whether  $r_J(\mathcal{I})\leq e_1(\mathcal{I})-e_0(\mathcal{I})+\ell(R/I_1)+1$ holds true.   In \cite{ms}, authors proved that 
$\widetilde{r}_J(\mathcal{I})\leq e_2(\mathcal{I})+1$. The following lemma improves the above bound in some cases. 
\begin{lemma}\label{4.6}
 Let $(R,\m)$ be a two dimensional Cohen-Macaulay local ring, $I$ an $\m$-primary ideal and $\mathcal{I}=\{I_n\}$ an $I$-admissible filtration. Suppose 
$I_1$ is an integrally closed ideal.  
Then, for a minimal reduction $J$ of $\mathcal{I}$, we have
$$
\widetilde{r}_J(\mathcal{I})\leq e_2(\mathcal{I})-e_1(\mathcal{I})+e_0(\mathcal{I})-\ell(R/I_1)+2.
$$
\end{lemma}
\begin{proof}
Since  $\ell(\widetilde{I_1}/J)=\ell (I_1/J) $, $e_i(\mathcal{I})=\widetilde{e_i}(\mathcal{I})$ for $0\leq i\leq2$,  and $\depth G(\widetilde{\mathcal{I}})\geq 1$, a similar calculation for filtration, as done in \eqref{the-equality-for-f-1}, will give
 \begin{align*}
   e_2(\mathcal{I})-e_1(\mathcal{I})+e_0(\mathcal{I})-\ell(R/I_1)=\sum_{n\geq 2}(n-1)\ell(\widetilde{I_{n+1}}/J\widetilde{I_n}).\nonumber
 \end{align*}
This gives $\ell(\widetilde{I_{n+1}}/J\widetilde{I_n})=0$ for all $n\geq  e_2(\mathcal{I})-e_1(\mathcal{I})+e_0(\mathcal{I})-\ell(R/I_1)+2$. Hence, $\widetilde{r}_J(\mathcal{I})\leq e_2(\mathcal{I})-e_1(\mathcal{I})+e_0(\mathcal{I})-\ell(R/I_1)+2.$
\end{proof}
\begin{proposition}\label{converse-of-theorem}
Let $(R,\m)$ be a three dimensional Cohen-Macaulay local ring and $I$ an $\m$-primary integrally closed ideal. Suppose $e_2(I)-e_1(I)+e_0(I)-\ell (R/I)\leq 1$. Then $e_3(I)=e_2(I)-e_1(I)+e_0(I)-\ell (R/I)$ if and only if the Ratliff-Rush filtration of $I$ behaves well modulo a superficial element $x\in I$.    
\end{proposition}
\begin{proof}
By \cite[Lemma 11]{itoh}, there exists a superficial element $x$ of $I$ so that $I/(x)$ is an integrally closed ideal of $R^\prime=R/(x)$. Set $I^\prime=IR^\prime$ and $J^\prime=JR^\prime$.  Then 
$e_2(I^\prime)-e_1(I^\prime)+e_0(I^\prime)-\ell (R^\prime/I^\prime)=e_2(I)-e_1(I)+e_0(I)-\ell(R/I)\leq 1.$ By Lemma \ref{4.6}, $\widetilde{r}_{J^\prime}(I^\prime)\leq 3$. This gives $\widetilde{e}_3(I^\prime)=\ell( \widetilde{(I^\prime)^3}/J^\prime\widetilde{(I^\prime)^2})$ using \eqref{formula-for-e3tilde} and $e_2(I)-e_1(I)+e_0(I)-\ell(R/I)=\ell( \widetilde{(I^\prime)^3}/J^\prime\widetilde{(I^\prime)^2})$ using \eqref{the-equality-for-f-1}. Hence, by Lemma \ref{lemma-value-of-e3},  $e_3(I)=\widetilde{e}_3(I^\prime)+b_I-s_I=\ell( \widetilde{(I^\prime)^3}/J^\prime\widetilde{(I^\prime)^2})+b_I-s_I=e_2(I)-e_1(I)+e_0(I)-\ell(R/I)+b_I-s_I$. 
Now $e_3(I)=e_2(I)-e_1(I)+e_0(I)-\ell (R/I)$ if and only if $b_I-s_I=0$ which is equivalent to the Ratliff-Rush filtration of $I$ behaving well modulo $x$
by part \eqref{remark-RR} of Remark \ref{remark-R}. 
\end{proof}
We end this section with the following remark. 
\begin{remark}\label{re4}
Suppose $d=3$ and the Ratliff-Rush filtration of $I$ behaves well modulo a superficial element. Then 
Rossi's bound holds, i.e., $r_J(I)\leq e_1(I)-e_0(I)+\ell(R/I)+1.$
\end{remark}
\begin{proof}
Let $\mathcal{F}=\{\widetilde{I^n}\}_{n\geq 0}$ and $J=(x_1,x_2,x_3)\subseteq I$ be a minimal reduction of $I$. We write $R^\prime=R/(x_1)$ and $\mathcal{F}^\prime=\{\frac{\widetilde{I^n}+(x_1)}{(x_1)}\}_{n\geq 0}$. Since $\widetilde{I^n}R^\prime= \widetilde{I^nR^\prime}$ for all $n\geq 0$, we get that $\depth G(\mathcal{F}^\prime)\geq 1$. This gives $\depth G(\mathcal{F})\geq 2.$ Therefore, $e_1(I)=e_1(\mathcal{F})=\mathop\sum\limits_{n\geq 0}v_n(\mathcal{F})$ by \cite[Theorem 4.7]{hm}.  Now using \eqref{Rossi's result}, we get that $r_J(I)\leq e_1(I)-e_0(I)+\ell(R/I)+1.$
\end{proof}
\section{Bounds for $e_3(I)$}\label{results-2}
As discussed earlier, Ratliff-Rush filtration of $I$ behaving well modulo a superficial sequence is a useful condition. However, it still involves all the Ratliff-Rush powers of $I$ and $I^\prime.$ Theorem \ref{theorem-marley-gen} provides a computable necessary condition for the Ratliff-Rush filtration of $I$ behaving well. 
In this section, we prove upper bounds for $e_3(I)$ for an integrally closed ideal $I$. Consequently, we obtain sufficient conditions for the Ratliff-Rush filtration of $I$ behaving well modulo a superficial element in dimension three. 
\begin{theorem}\label{T_{11}}
Let $(R,\m)$ be a Cohen-Macaulay local ring of dimension $d\geq 3$, $I$ an $\m$-primary integrally closed ideal and $J$ a minimal reduction of $I$. Then 
\begin{equation}\label{10}
    e_3(I)\leq  \frac{(r_J(I)-1)}{2}\big(e_2(I)-e_1(I)+e_0(I)-\ell ({R}/{I})\big).
\end{equation}
Furthermore, suppose equality holds in \eqref{10} and $d=3$, then the Ratliff-Rush filtration of $I$ behaves well modulo a superficial element. 
Conversely, if Ratliff-Rush filtration of $I$ behaves well modulo superficial sequence $x_1,\ldots,x_{d-2} \in I$ and $r_J(I)\leq 3$, then equality holds in \eqref{10}. 
\end{theorem}
\begin{proof}
 Suppose $d=3$. Let $x\in I$ be a superficial element such that $I/(x)$ is an integrally closed ideal of $R^\prime=R/(x)$. Set $I^\prime=IR^\prime$ and $J^\prime=JR^\prime$, then $\dim R^\prime=2$ and $e_i(I)=e_i(I^\prime)=\widetilde{e_i}(I^\prime)$ for $0\leq i\leq 2$. Since $I^\prime$ is integrally closed, we have $\ell(\widetilde{I^\prime}/J^\prime)=\ell(I^\prime/J^\prime)$ which gives equality in \eqref{the-equality-for-f-1}, i.e., 
\begin{align}
e_2(I)-e_1(I)+e_0(I)-\ell(R/I)
 =\sum_{n\geq 2}(n-1)\ell (\widetilde{{I^\prime}^{n+1}}/J^\prime\widetilde{{I^\prime}^n}).\nonumber
\end{align}
Note that $\ell (\widetilde{{I^\prime}^{n+1}}/J^\prime\widetilde{{I^\prime}^n})=0$ for $n\geq \widetilde{r}_{J^\prime}(I^\prime)$. 
Now using \eqref{the-equality-for-f-1} and \eqref{formula-for-e3tilde}, we get
\begin{align}
\widetilde{e_3}(I^\prime)-\frac{(r_J(I)-1)}{2}\big(e_2(I)-e_1(I)+e_0(I)-\ell(R/I)\big)&=\sum_{n \geq 2}^{\widetilde{r}_{J^\prime}(I^\prime)-1}\Big[{n \choose 2}-\frac{(r_J(I)-1)(n-1)}{2}\Big]\ell (\widetilde{{I^\prime}^{n+1}}/J^\prime\widetilde{{I^\prime}^n})\nonumber\\
\begin{split}
&=\sum_{n \geq 2}^{\widetilde{r}_{J^\prime}(I^\prime)-1}\frac{(n-1)}{2}(n-(r_J(I)-1))\ell(\widetilde{{I^\prime}^{n+1}}/J^\prime\widetilde{{I^\prime}^n})\label{kf-less-than-0-1}\\
&\leq 0
\end{split}
\end{align}
since $\widetilde{r}_{J^\prime}(I^\prime)\leq r_{J^\prime}(I^\prime)\leq r_J(I)$ by \cite[Theorem 2.1]{ms} and so $ (n-1)(n-(r_J(I)-1))\leq 0$ for each $2\leq n \leq \widetilde{r}_{J^\prime}(I^\prime)-1$. Now by Lemma \ref{lemma-value-of-e3}, 
\begin{align}\label{15}
e_3(I)=\widetilde{e_3}(I^\prime)+(b_I-s_I)\leq \widetilde{e_3}(I^\prime)\leq \frac{(r_J(I)-1)}{2}\big(e_2(I)-e_1(I)+e_0(I)-\ell(R/I)\big).
\end{align}
Suppose
$ e_3(I)=
    \Big(\frac{r_J(I)-1}{2}\Big)\big(e_2(I)-e_1(I)+e_0(I)-\ell ({R}/{I})\big)$. Then by \eqref{15}, we get   $\widetilde{e_3}(I^\prime)+(b_I-s_I)= \widetilde{e_3}(I^\prime)$ which implies $ b_I-s_I=0.$ Therefore, by part \eqref{remark-RR} of Remark \ref{remark-R}, the Ratliff-Rush filtration of $I$ behaves well modulo $x$.\\
Now suppose $d\geq 4$ and
 $x \in I$ is a superficial element  such that $I^\prime=I/(x)$ is integrally closed. 
 Since $r_{J^\prime}(I^\prime)\leq r_J(I)$, by induction hypothesis, we have
\begin{align}
e_3(I)=e_3(I^\prime)&\leq \frac{(r_{J^\prime}(I^\prime)-1)}{2}\big(e_2(I^\prime)-e_1(I^\prime)+e_0(I^\prime)-\ell ({R^\prime}/{I^\prime})\big)\nonumber \\
&\leq \frac{(r_J(I)-1)}{2}\big(e_2(I)-e_1(I)+e_0(I)-\ell ({R}/{I})\big)\nonumber.
\end{align}
 Suppose $r_J(I)\leq 3$ and the Ratliff-Rush filtration of $I$ behaves well modulo superficial sequence $x_1,\ldots, x_{d-2}$. Then, using Theorem \ref{theorem-marley-gen} and \eqref{10}, we get $e_3(I)=
    \Big(\frac{r_J(I)-1}{2}\Big)\big(e_2(I)-e_1(I)+e_0(I)-\ell ({R}/{I})\big)$.
\end{proof}

We prove two more bounds for $e_3(I)$ for an integrally closed ideal $I.$ 
 
\begin{theorem}\label{T-14}
Let $(R,m)$ be a Cohen-Macaulay local ring of dimension $d\geq 3$ and $I$ an $\m$-primary integrally closed ideal. Then 
\begin{align}\label{9}
e_3(I)\leq \frac{\big(e_1(I)-e_0(I)+\ell(R/I)\big)}{2} \big(e_2(I)-e_1(I)+e_0(I)-\ell ({R}/{I})\big).\end{align}
Furthermore, suppose $d=3$ and  equality holds in \eqref{9}, then the Ratliff-Rush filtration of $I$ behaves well modulo a superficial element. Conversely, if Ratliff-Rush filtration of $I$ behaves well modulo a superficial sequence $x_1,\ldots,x_{d-2} \in I$ and $e_1(I)-e_0(I)+\ell(R/I)\leq 2$, then equality holds in \eqref{9}.
\end{theorem}
\begin{proof}
 Suppose $d=3$. Let $x\in I$ be a superficial element such that $I/(x)$ is an integrally closed ideal of $R^\prime$. As in \eqref{kf-less-than-0-1}, we have 
\begin{align}
\widetilde{e_3}(I^\prime)-\frac{(\widetilde{r}_{J^\prime}(I^\prime)-1)}{2}\big(e_2(I)-e_1(I)+e_0(I)-\ell(R/I)\big)&=\sum_{n \geq 2}^{\widetilde{r}_{J^\prime}(I^\prime)-1}\Big[{n \choose 2}-\frac{(\widetilde{r}_{J^\prime}(I^\prime)-1)(n-1)}{2}\Big]\ell (\widetilde{{I^\prime}^{n+1}}/J^\prime\widetilde{{I^\prime}^n})\nonumber\\
\begin{split}
&=\sum_{n \geq 2}^{\widetilde{r}_{J^\prime}(I^\prime)-1}\frac{(n-1)}{2}(n-(\widetilde{r}_{J^\prime}(I^\prime)-1))\ell (\widetilde{{I^\prime}^{n+1}}/J^\prime\widetilde{{I^\prime}^n})\label{tilde_e_3}\\
&\leq 0 
\end{split}
 \end{align}
 since $(n-1)(n-(\widetilde{r}_{J^\prime}(I^\prime)-1))\leq 0$ for each $2\leq n\leq \widetilde{r}_{J^\prime}(I^\prime)-1.$ By \cite[Theorem 2.1]{ms}, we have $\widetilde{r}_{J^\prime}(I^\prime)\leq r_{J^\prime}(I^\prime).$ 
Therefore, 
\begin{align}
\widetilde{e}_3(I^\prime)&\leq 
\frac{(\widetilde{r}_{J^\prime}(I^\prime)-1)}{2}\big(e_2(I)-e_1(I)+e_0(I)-\ell(R/I)\big)\nonumber\\
\begin{split}
&\leq \frac{(r_{J^\prime}(I^\prime)-1)}{2}\big(e_2(I)-e_1(I)+e_0(I)-\ell(R/I)\big)\label{18}\\
&\leq  \frac{e_1(I^\prime)-e_0(I^\prime)+\ell (R^\prime/I^\prime)}{2}\big(e_2(I)-e_1(I)+e_0(I)-\ell(R/I)\big)\\
&=\frac{e_1(I)-e_0(I)+\ell(R/I) }{2}\big(e_2(I)-e_1(I)+e_0(I)-\ell(R/I)\big)
\end{split}
\end{align} 
where the last inequality follows from Rossi's bound in dimension two \cite[Corollary 1.5]{r}. Finally, to get the result, we use $e_3(I)=\widetilde{e}_3(I^\prime)+b_I-s_I\leq \widetilde{e_3}(I^\prime)$ as in the proof of  Theorem \ref{T_{11}}.

Suppose $e_3(I)=\frac{(e_1(I)-e_0(I)+\ell(R/I))}{2} \Big(e_2(I)-e_1(I)+e_0(I)-\ell ({R}/{I})\Big) $. Then, by \eqref{18}, $\widetilde{e_3}(I^\prime)+(b_I-s_I)= \widetilde{e_3}(I^\prime)$ which implies $ b_I-s_I=0.$ Therefore by part \eqref{remark-RR} of Remark \ref{remark-R}, the Ratliff-Rush filtration of $I$ behave well modulo $x$.

Now suppose $d\geq 4$ and $x \in I$ is a superficial element  such that $I^\prime=I/(x)$ is integrally closed. 
 Then, by induction hypothesis, we get
$$
e_3(I)=e_3(I^\prime)\leq \frac{(e_1(I^\prime)-e_0(I^\prime)+\ell(R^\prime/I^\prime))}{2} \Big(e_2(I^\prime)-e_1(I^\prime)+e_0(I^\prime)-\ell ({R}/{I})\Big)
$$
$$
=\frac{(e_1(I)-e_0(I)+\ell(R/I))}{2} \Big(e_2(I)-e_1(I)+e_0(I)-\ell ({R}/{I})\Big).
$$


  Suppose $(e_1(I)-e_0(I)+\ell(R/I))\leq 2 $  and the Ratliff-Rush filtration of $I$ behaves well modulo a superficial sequence $x_1,\ldots,x_{d-2} \in I$. Then, using Theorem \ref{theorem-marley-gen} and \eqref{9}, we get
  \begin{equation*}
  e_3(I)= \frac{\big(e_1(I)-e_0(I)+\ell(R/I)\big)}{2} \big(e_2(I)-e_1(I)+e_0(I)-\ell ({R}/{I})\big).\qedhere
  \end{equation*}
\end{proof}
\begin{corollary}\label{3.4}
Let $(R,m)$ be a Cohen-Macaulay local ring of dimension $d\geq 3$ and $I$ an $\m$-primary integrally closed ideal. Then
\begin{align}\label{16} 
e_3(I)\leq \frac{\big(e_2(I)-1\big)}{2} \big(e_2(I)-e_1(I)+e_0(I)-\ell ({R}/{I})\big).\end{align}
Furthermore, suppose $d=3$ and equality holds in \eqref{16}, then Ratliff-Rush filtration of $I$ behaves well modulo a superficial element. Conversely, if Ratliff-Rush filtration of $I$ behaves well modulo a superficial sequence $x_1,\ldots,x_{d-2}$ and $e_2(I)\leq 3$, then equality holds in \eqref{16}.
\end{corollary}
\begin{proof}
We may assume that 
$e_2(I)-e_1(I)+e_0(I)-\ell (R/I)\geq 1$. Otherwise, $e_2(I)-e_1(I)+e_0(I)-\ell (R/I)=0$ and $e_3(I)\leq 0$ 
 by Theorem \ref{T_{11}}. We have 
$e_1(I)-e_0(I)+\ell (R/I) \leq e_2(I)-1$. Hence, the inequality in \eqref{16} follow from Theorem \ref {T-14}. \\
Further, suppose $d=3.$ Then, using \eqref{18} and Lemma \ref{lemma-value-of-e3},  we have
\begin{align}
e_3(I)= \widetilde{e}_3(I)+b_I-s_I\leq \widetilde{e}_3(I)\leq \frac{(e_2(I)-1)}{2}(e_2(I)-e_1(I)+e_0(I)-\ell (R/I)).\label{e}
\end{align}
Now, suppose $e_3(I)= \frac{(e_2(I)-1)}{2} (e_2(I)-e_1(I)+e_0(I)-\ell ({R}/{I}))$. Then by \eqref{e}, $\widetilde{e}_3(I^\prime)+(b_I-s_I)=\widetilde{e}_3(I^\prime)$ which implies $b_I-s_I=0$. Therefore, by part \eqref{remark-RR} of Remark \ref{remark-R}, the Ratliff-Rush filtration of $I$ behaves well modulo superficial element. For the  converse, suppose $e_2(I)\leq 3$ and the Ratliff-Rush filtration of $I$ behaves well modulo superficial sequence $x_1,\ldots,x_{d-2}$. Then using Theorem  \ref{theorem-marley-gen} and \eqref{16}, we get the result.
\end{proof}
In \cite[Proposition 6.4 ]{Put2} and \cite[Proposition 2.6]{mf2}, authors proved that for an $\m$-primary ideal $I$, $e_2(I)\leq 1$ implies $e_3(I)\leq 0.$ For an integrally closed $\m$-primary ideal $I$, we obtain the following corollary. 
\begin{corollary}
 Let $(R,\m)$ be a Cohen-Macaulay local ring of dimension $d\geq 3$ and $I$ an $\m$-primary integrally closed ideal. Then $e_3(I)\leq (e_2(I)-1)^2/2.$
\end{corollary}
\begin{proof}
Suppose $e_1(I)-e_0(I)+\ell(R/I)=0$. Then $G(I)$ is Cohen-Macaulay and $e_3(I)=0\leq {(e_2(I)-1)}^2/2.$ Now let $e_1(I)-e_0(I)+\ell(R/I)\geq 1$. Then by Corollary \ref{3.4}, 
\begin{align*}
    e_3(I)\leq \frac{(e_2(I)-1)(e_2(I)-1)}{2}.\qedhere
\end{align*}
\end{proof}
\begin{remark}
Suppose $d=3$ and equality holds in \eqref{9}, then 
equality holds in \eqref{10}. Converse is not true as evident by Example \ref{eg2}.
\begin{proof}
 Suppose equality holds in \eqref{9}, then the Ratliff-Rush filtration of $I$ behaves well modulo superficial element. Hence by Remark \ref{re4} and Theorem \ref{T_{11}}, we get 
 \begin{align*}
 e_3(I)\leq& \frac{(r_J(I)-1)}{2}\big(e_2(I)-e_1(I)+e_0(I)-\ell(R/I)\big)\\
 &\leq \frac{e_1(I)-e_0(I)+\ell(R/I)}{2}\big(e_2(I)-e_1(I)+e_0(I)-\ell(R/I)\big)\\
 &=e_3(I).\qedhere
 \end{align*} 
\end{proof}
\end{remark}
The bound in Theorem \ref{T_{11}} is sharp as evident by the following example. 
\begin{example}\cite[Example 3.2]{cpr}\label{eg2}
Let $R = K[[x,y,z]]$, where $x,y,z$ are indeterminates and $K$ is a field of characteristic $\neq 3$. Let $N=(x^4,x(y^3+z^3),y(y^3+z^3),z(y^3+z^3))$ and set $I = N+\m^5,$ where $\m$ is the maximal ideal of $R$. The ideal $I$ is a normal $\m$-primary ideal, $r_J(I)=3$ for any minimal reduction $J$ of $I$ and we have
$$
H(I, t)=\frac{31+43t+t^2+t^3}{{(1-t)}^3}$$
which gives $e_2(I)-e_1(I)+e_0(I)-\ell(R/I)=1=e_3(I).$ Thus we have equality in Theorem \ref{T_{11}} which gives 
that the Ratliff-Rush filtration of $I$ behaves well modulo superficial element
where as $\frac{(e_1(I)-e_0(I)+\ell(R/I))}{2}\big(e_2(I)-e_1(I)+e_0(I)-\ell ({R}/{I})\big)=3/2>e_3(I)$  and $\frac{(e_2(I)-1)}{2}\big(e_2(I)-e_1(I)+e_0(I)-\ell ({R}/{I})\big)=3/2>e_3(I).$ Note that  $e_1(I)-e_0(I)+\ell(R/I)=3$ and $e_2(I)=4$.
\end{example}
We recall Example \ref{eg1} to note that the bounds given in Theorem \ref{T_{11}}, Theorem \ref{T-14} and Theorem \ref{3.4} may not hold for non-integrally closed $\m$-primary ideals.
\begin{example}\label{example 4.7}
 We have  $\frac{(r_J(I)-1)}{2}\big(e_2(I)-e_1(I)+e_0(I)-\ell ({R}/{I})\big)=-15$,   $\frac{(e_1(I)-e_0(I)+\ell(R/I))}{2}\big(e_2(I)-e_1(I)+e_0(I)-\ell ({R}/{I})\big)=-70$ and $\frac{(e_2(I)-1)}{2}\big(e_2(I)-e_1(I)+e_0(I)-\ell ({R}/{I})\big)=-15$ where as $e_3(I)=0.$
\end{example}
Now we gather the cases where $e_3(I)\leq e_2(I)-e_1(I)+e_0(I)-\ell(R/I)$ holds. In part \eqref{part-2}, we recover  \cite[Theorem 2.8]{mf2}.
\begin{corollary}\label{final-conclusion}
Let $(R,\m)$ be a Cohen-Macaulay local ring of dimension $d\geq 3$ and $I$ an $\m$-primary integrally closed ideal. Then $e_3(I)\leq e_2(I)-e_1(I)+e_0(I)-\ell(R/I)$ if one of the following conditions hold: 
\begin{enumerate}
    \item $r_J(I)\leq 3$.
    \item\label{part-2} $e_2(I)-e_1(I)+e_0(I)-\ell (R/I)\leq  1.$
    \item $e_1(I)-e_0(I)+\ell(R/I)\leq 2.$
    \item $e_2(I)\leq 3.$
\end{enumerate}
Furthermore, if $e_1(I)-e_0(I)+\ell (R/I)=2$, then $e_3(I)= e_2(I)-e_1(I)+e_0(I)-\ell ({R}/{I}).$
\end{corollary}
\begin{proof}
\begin{enumerate}
    \item  It follows from Theorem \ref{T_{11}} since $(r_J(I)-1)/2\leq 1$.
    \item  Suppose $d=3$. Let $x$ be a superficial element of $I$ such that $I^\prime=I/(x)$ is integrally closed in $R^\prime=R/(x)$. Then, by Lemma \ref{4.6}, $\widetilde{r}_{J^\prime}(I^\prime)\leq 3$. Hence, by Lemma \ref{lemma-value-of-e3} and  \eqref{tilde_e_3},  we get $e_3(I)\leq \widetilde{e}_3(I^\prime)\leq e_2(I)-e_1(I)+e_0(I)-\ell(R/I).$ For $d\geq 4$, the inequality follows by induction on $d$. 
    \item It follows from the bound in Theorem \ref{T-14}. 
    \item It follows from the bound in Corollary \ref{3.4}. 
\end{enumerate}
For the last part, suppose $e_1(I)-e_0(I)+\ell (R/I)=2$. Then, by \cite [Theorem 4.5]{or}, $\depth G(I)\geq d-1$ and $I^4=JI^3$. 
Hence by Corollary \ref{cor-mar-gen}, $e_3(I)\geq e_2(I)-e_1(I)+e_0(I)-\ell({R}/{I})$. 
\end{proof}
\section{Applications}\label{applications}
In this section, we write some interesting consequences of the theorems of last sections. \\

{\bf{I}}. In \cite[Theorem 6.2]{Put2}, Puthenpurakal proved that $e_2(I)=e_3(I)=\cdots=e_d(I)=0$ implies that the Ratliff-Rush filtration of $I$ behaves well modulo a superficial sequence. However, when $I$ is integrally closed, we know that $e_2(I)=0$ implies that $G(I)$ is Cohen-Macaulay by results of Itoh \cite[Corollary 13]{itoh}. In this context, we discuss two examples below. The first one demonstrate the case when $e_2(I)\neq 0$, $\depth G(I)=0$ but the Ratliff-Rush filtration of $I$ behaves well. The second example is an application of Theorem \ref{theorem-marley-gen} as well as Corollary \ref{3.4}.
\begin{example}\cite[Theorem 5.2]{or}\label{eg4}
Let $m\geq 0$, $d\geq 2$ and $k$ be an infinite field. Consider the power series ring $D=k[[\{X_j\}_{1\leq j\leq m}, Y, \{V_j\}_{1\leq j\leq d}, \{Z_j\}_{1\leq j\leq d}]]$ with $m+2d+1$ indeterminates and the ideal $\mathfrak{a}=[(X_j~|~1\leq j\leq m)+(Y)].[(X_j~|~1\leq j\leq m)+(Y)+(V_i~|~1\leq i\leq d)]+(V_iV_j~|~1\leq i,j\leq d, i\neq j)+(V_i^3-Z_iY~|~1\leq i\leq d).$ Define $R=D/\mathfrak{a}$ and $x_i,y,v_i,z_i$ denote the images of $X_i, Y, V_i, Z_i$ in $R$ respectively. Let $\mathfrak{m}=(x_j~|~1\leq j\leq m)+(y)+(v_j~|~1\leq j\leq d)+(z_j~|~1\leq j\leq d)$ be the maximal ideal in $R$ and $Q=(z_j~|~1\leq j\leq d).$ Then \begin{enumerate}
    \item $R$ is Cohen-Macaulay local ring with $\dim R=d,$
    \item $Q$ is a minimal reduction of $\m$ with $r_Q(\m)=3,$
    \item $\depth G(\m)=0,$ 
    \item $e_0(\m)=m+2d+2;$ $e_1(\m)=m+3d+2$; $e_2(\m)=d+1$ and $e_i(\m)=0$ for $3\leq i\leq d.$ 
    \end{enumerate}
    Thus $e_3(\m)=0=\frac{\big(r_Q(\m)-1\big)}{2}\big(e_2(\m)-e_1(\m)+e_0(\m)-\ell (R/\m)\big)$. When $d=3$, the Ratliff-Rush filtration of $\m$ behaves well modulo superficial element by Theorem \ref{T_{11}}.
\end{example} 
\begin{example}\cite[Example 3.5]{cpr}\label{eg-2-sec-6}
Let $R = k[[X,Y,Z,U,V,W]]/(Z^2,ZU,ZV,UV,YZ-U^3,XZ-V^3)$ be a three dimensional Cohen-Macaulay local ring and $X,Y,Z,U,V,W$ indeterminates. Let $x,y,z,u,v,w$ denote the corresponding images of $X,Y,Z,U,V,W$ in $R$ and $m=(x,y,z,u,v,w)$. Then $G(\m)$ has depth $1$ and 
$$
H(\m,t)= \frac{1+3t+3t^3-t^4}{{(1-t)}^3}
$$
which gives $e_2(\m)=3,e_1(\m)=8,e_0(\m)=6,\ell(R/\m)=1,e_3(\m)=-1$. Thus  $e_2(\m)-e_1(\m)+e_0(\m)-\ell (R/\m)=0$ and $e_3(\m)=-1.$
Therefore the Ratliff-Rush filtration of $\m$ does not behave well modulo superficial element by Theorem \ref{theorem-marley-gen} or by Corollary \ref{3.4}.
\end{example}
In \cite[Theorem 3.5]{Put2}, it was proved that if $d=2$, $I$ is integrally closed and $e_2(I)-e_1(I)+e_0(I)-\ell(R/I)=0$, then Ratliff-rush filtration of $I$ behaves well modulo superficial element. Example \ref{eg-2-sec-6} shows that the above result can not be extended in dimension three. \\

{\bf{II.}} We now prove a result on the signature of $e_4(I)$. In \cite{e}, Elias proved that $\depth G(I^n)$ is constant for $n\gg0$.
In the first part of Theorem \ref{thm-e4}, we give a necessary condition for $\depth G(I^n) \geq 3$ for $n\gg 0.$ The second part is a generalization of \cite[Theorem 2.5]{mf2} which was the case where $I$ is an asymptotically normal ideal. 
\begin{theorem}\label{thm-e4}
Let $(R,\m)$ be a Cohen-Macaulay local ring of dimension $d=4$. Let $I$ be an $\m$-primary ideal of $R$ and $J$ a minimal reduction of $I$. 
\begin{enumerate}
    \item Suppose $\depth G(I^q)\geq 3$ for some $q\geq \eta(I)$. Then 
    $e_4(I)\geq 0.$
\item Suppose $I^q$ is integrally closed for some $q\geq \eta(I)$ and $r_J(I)\leq 3$. Then $e_4(I)\leq 0$.
\end{enumerate}
\end{theorem}
\begin{proof}
 Let $q\geq \eta(I)$. We have   
\begin{equation}\label{hs1}
    \ell(R/I^n )= e_0(I){n+3 \choose 4}-e_1(I){n+2 \choose 3}+e_2(I){n+1 \choose 2}-e_3(I){n \choose 1}+e_4(I)
\end{equation}
for $n\geq \eta(I)$ and  
\begin{equation}\label{hs2}
 \ell(R/(I^q)^n )= \epsilon_0{n+3 \choose 4}-\epsilon_1{n+2 \choose 3}+\epsilon_2{n+1 \choose 2}-\epsilon_3{n \choose 1}+\epsilon_4   
\end{equation}
for $n\gg 0$ where $\epsilon_i=e_i(I^q)$ for $0\leq i\leq 4$. Since $\ell (R/(I^q)^n)=\ell (R/I^{qn})$, using  \eqref{hs1} and \eqref{hs2}, we get   
\begin{align*}
\ell(R/I^{qn})&=e_0(I){nq+3 \choose 4}-e_1(I){nq+2 \choose 3}+e_2(I){nq+1 \choose 2}-e_3(I){nq \choose 1}+e_4(I)\\
&=\epsilon_0{n+3 \choose 4}-\epsilon_1{n+2 \choose 3}+\epsilon_2{n+1 \choose 2}-\epsilon_3{n \choose 1}+\epsilon_4 
\end{align*}
for $n\gg 0$. On comparing the coefficients, we get 
\begin{align*}
 \epsilon_0=e_0(I)q^4,~ \epsilon_1=\frac{3}{2}e_0(I)(q^4-q^3)+e_1(I)q^3,~\epsilon_2=\frac{e_0(I)}{12}(11q^2+7q^4-18q^3)+e_1(I)(q^3-q^2)+e_2(I)q^2,\\
\epsilon_3=e_0(I){q \choose 4}+e_1(I){q \choose 3}+e_2(I){q \choose 2}+e_3(I){q \choose 1} \mbox{ and }  \epsilon_4=e_4(I).
\end{align*}
 Therefore, 
\begin{align*}
     \epsilon_3-\epsilon_2+\epsilon_1-\epsilon_0+\ell (R/I^q)=& -e_0(I){q+3 \choose 4}+e_1(I){q+2 \choose 3}-e_2(I){q+1 \choose 2}+e_3(I){q \choose 1}+\ell (R/I^q)\\
     =&e_4(I).
\end{align*}
since $q\geq \eta(I)$. Now suppose $\depth G(I^q)\geq 3.$ Then $e_4(I)=\epsilon_3-(\epsilon_2-\epsilon_1+\epsilon_0-\ell(R/I^q))\geq 0$ by Corollary \ref{cor-mar-gen}. 

For the second part, suppose $I^q$ is integrally closed. 
Using Theorem \ref{T_{11}}, we have
\begin{align*}
 e_4(I)=&\epsilon_3-(\epsilon_2-\epsilon_1+\epsilon_0-\ell(R/I^q))\\
 &\leq \frac{(r(I^q)-3)}{2}\Big(\epsilon_2-\epsilon_1+\epsilon_0-\ell(R/I^q)\Big)\\
 &\leq \frac{1}{2}\Big(\ceil{(r(I)-3)/q}\Big)(\epsilon_2-\epsilon_1+\epsilon_0-\ell(R/I^q))
\end{align*} 
 where the last inequality follows from \cite[Lemma 2.7]{h}. Here, $\ceil{x}:=\min\{m\in\mathbb{Z}~|~m\geq x\}.$ Now if $r_J(I)\leq 3$, then $\ceil{{r(I)-3}/q}\leq 0$ and the result follows. 
\end{proof}
\begin{example} \cite[Example 3.1]{mf2}
Let $R=k[x,y,z,u]_{(x,y,z,u)}$ and $I =(x^3,y^3,z^3,u^3,xy^2,yz^2,zu^2,xyz,xyu).$ Then $J=(x^3,y^3,z^3,u^3)$ is a minimal reduction of $I$ with $r_J(I)=4$. Moreover,  $\depth G(I) =2,$ the Hilbert series is
$$
H(I,t)=\frac{33+19t+21t^2+7t^3+5t^4-3t^5-t^6}{(1-t)^4}
$$
and the Hilbert polynomial is
$$
P_I(n)=81{n+3 \choose 4}-81{n+2 \choose3}+27{n+1\choose 2}+23{n\choose1}-50.
$$
Therefore, $e_0(I)=81,~e_1(I)=81,~e_2(I)=27,
~e_3(I)=-23,~e_4(I)=-50.$ Hence, using Theorem \ref{thm-e4}, we have $2=\depth G(I)\leq \depth G(I^n)\leq 2$ for $n\gg 0$. 
\end{example}
A similar calculation in dimension three, as done in the proof of Theorem \ref{thm-e4},  gives the equality 
\begin{equation}\label{eqn-e3}
    e_3(I)=e_2(I^q)-e_1(I^q)+e_0(I^q)-\ell(R/I^q)
\end{equation}
 for $q\geq \eta(I)$. We refer to the proof of Theorem 4.1 in  \cite{cpr} for detailed calculation. We partially recover  \cite[Corollary 4.5]{cpr} below. 
\begin{proposition}\label{cpr recover}
Let $(R,\m)$ be a Cohen-Macaulay local ring of dimension $d=3$ and $I$ an $\m$-primary  ideal. Suppose $I^q$ is integrally closed for some $q\geq \eta(I)$ and $r(I^q)\leq 2$. Then $e_3(I)=0$. 
\end{proposition} 
\begin{proof}
By \eqref{eqn-e3} and Theorem \ref{T_{11}}, 
\begin{align*}
e_3(I)=e_3(I^q)&\leq \frac{(r(I^q)-1)}{2}\Big(e_2(I^q)-e_1(I^q)+e_0(I^q)-\ell(R/I^q)\Big)=\frac{(r(I^q)-1)}{2} e_3(I).
\end{align*}
Suppose $e_3(I)\neq 0$. Then we get  $r(I^q)\geq 3$ which is a contradiction.  
\end{proof}

{\bf{III.}} Next, we prove an upper bound for the reduction number $r_J(I)$ of an integrally closed ideal $I$. If $\depth G(I)\geq d-2,$ then by \cite[Corollary 1.5]{r}, we have
\begin{equation}
r_J(I)\leq e_1(I)-e_0(I)+\ell(R/I)+1. \nonumber  
\end{equation}
We prove the following result. 
\begin{theorem}\label{improvment ms result}
Let $(R,\m)$ be a Cohen-Macaulay local ring of dimension $d\geq 3$ and $I$ an integrally closed $\m$-primary ideal. Suppose $\depth G(I)\geq d-3$. Then 
\begin{align}
    r_J(I)\leq e_1(I)-e_0(I)+\ell (R/I)+1+e_2(I)(e_2(I)-e_1(I)+e_0(I)-\ell (R/I))-e_3(I).\nonumber
\end{align}
\end{theorem}
 Our proof is an improvement of the methods discussed in \cite[Theorem 4.1]{ms}.
We briefly recall some ideas from \cite{ms}. As in \cite{gr}, the second Hilbert function of $\mathcal I$ is defined as $H^2_{\mathcal I}(n)=\mathop\sum\limits_{i=0}^n H_{\mathcal I}(i)$ and the second Hilbert polynomial, denoted by $P^2_{\mathcal I}(n)$ is the polynomial which coincides with $H^2_{\mathcal I}(n)$ for large values of $n$. Let $\underline{x}$ denote a sequence $x_1,\ldots,x_m$ in $R$ and $n$ be an integer. We define modified Koszul complex $C_.(\underline x,\mathcal{I},n)$ as in \cite{hm} by induction on $m$: For $m=1,$  
  $C_.(x_1,\mathcal{I},n)$ is the complex $$0\longrightarrow R/I_{n-1}\overset{x_1}\longrightarrow R/I_n\longrightarrow 0$$  
  and for $m>1$,  consider the chain map $f: C_.(x_1,\ldots,x_{m-1},\mathcal{I},n-1)\longrightarrow C_.(x_1,\ldots,x_{m-1},\mathcal{I},n)$ defined as multiplication by $x_m$. Define
$C_.(x_1,\ldots,x_{m},\mathcal{I},n)$ to be the mapping cylinder of $f$. Then  $C_.(x_1,\ldots,x_{m},\mathcal{I},n)$
has the following form
$$0\longrightarrow R/{I_{n-m}}\longrightarrow (R/{I_{n-m+1}})^m\longrightarrow (R/{I_{n-m+2}})^{{m\choose 2}}
\longrightarrow \quad ...\quad \longrightarrow R/{I_n}\longrightarrow 0.$$
For convenience of notation, we write ${C_.}(\mathcal{I},n)$ $ = {C_.}({x_1},{x_2},...,{x_m},\mathcal{I},n)$ and ${C_.}^\prime (\mathcal{I},n) = {C_.}({x_1},{x_2},...,{x_{m - 1}},\mathcal{I},n)$ when the value of $m$ is clear. For any integer $n$, there is an
exact sequence of complexes
$$0\longrightarrow {C_.}^\prime(\mathcal{I},n)\longrightarrow {C_.}(\mathcal{I},n)\longrightarrow{C_.}^\prime (\mathcal{I},n-1)[- 1]\longrightarrow 0$$
which induces a long exact sequence of homology modules
\begin{equation}
\cdots\longrightarrow {H_i}({C_.}^\prime(\mathcal{I},n))\longrightarrow {H_i}({C_.}(\mathcal{I},n))\longrightarrow {H_{i-1}}({C_.}^\prime (\mathcal{I},n - 1)){\kern 1 pt}\; \stackrel{\underline{+}x_m}{\longrightarrow} {H_{i - 1}}({C_.}^\prime (\mathcal{I},n)) \longrightarrow\cdots. \nonumber
\end{equation}
 We recall the following formula from \cite[\S 4.]{hm}. For each $n$ and $1\leq i \leq d$,
\begin{equation}\label{eq-ei}
{e_i}(\mathcal{I})=\sum\limits_{n=i}^\infty{\left({{}_{i-1}^{n-1}} \right)}\;{\Delta ^d}[{P_\mathcal{I}}(n-d)-{H_\mathcal{I}}(n-d)] \;\mbox{ and }\end{equation}
\begin{equation}\label{eq-delta}
{\Delta^d}[{P_\mathcal{I}}(n-d)-{H_\mathcal{I}}(n-d)]=\ell({{I_{n}}}/{J{I_{n-1}}})-\sum\limits_{i=2}^d {{{(-1)}^i}\ell({H_i}({C_.}(\mathcal{I},n)))}\end{equation}
where $J=(x_1,\ldots,x_d)$ is a minimal reduction of $\mathcal{I}$ and ${H_i}({C_.}(\mathcal{I},n))={H_i}({C_.}(x_1,\ldots,x_d,\mathcal{I},n)).$

\begin{proof}[Proof of Theorem \ref{improvment ms result}]
Suppose $d=3$. Let $\mathcal F=\{\widetilde{I^n}\}$ and $x\in I$ be a superficial element for $I$ such that $I^\prime$ is integrally closed. Then $x$ is also superficial for the filtration $\mathcal F$. Suppose $J=(x,y,z)$ is a minimal reduction of $I$. Let $R^\prime=R/(x)$ and ${\mathcal F}^\prime=\{{\mathcal F}^\prime_n=\frac{\widetilde{I^n}+(x)}{(x)}\}$. Since $\depth G(\mathcal F)\geq 1$, we have $v_n(\mathcal F)=v_n({\mathcal F}^\prime)$.
By using \eqref{eq-ei} and \eqref{eq-delta}, we get 
\begin{eqnarray}
e_1({\mathcal F}^\prime)
&=& \sum_{n\geq 0}v_n({\mathcal F}^\prime)-\sum_{n\geq 1}\ell(H_2(C_.({\mathcal F}^\prime,n))) .\label{eq}
\end{eqnarray}
Since $x$ is a superficial element for $\mathcal F$,  $e_1(I)=e_1(\mathcal F)=e_1({\mathcal F}^\prime)$.
Hence, by using \eqref{Rossi's result} and (\ref{eq}), we get 
\begin{equation}\label{general-stat2}
    r_J(I)\leq e_1(I)+\sum_{n\geq 1}\ell(H_2(C_.({\mathcal F}^\prime,n)))-e_0(I)+\ell(R/I)+1.
    \end{equation}
By \cite[Lemma 3.2]{hm}, we have  $H_2(C.({\mathcal F}^\prime,n))=\frac{{\mathcal F}^\prime_{n-1}:( y^\prime, z^\prime)}{{\mathcal F}^\prime_{n-2}}$.
Since ${\mathcal F}^\prime_{n-1}:( y^\prime, z^\prime)\subseteq \widetilde{{\mathcal F}^\prime_{n-2}}$,  $$\ell(H_2(C.({\mathcal F}^\prime,n)))\leq \ell\left (\frac{\widetilde{{\mathcal F}^\prime_{n-2}}}{{\mathcal F}^\prime_{n-2}}\right ).$$
Therefore, for $m\gg 0$, we have
\begin{eqnarray*}
0\leq \sum_{n= 0}^m \ell(H_2(C.({\mathcal F}^\prime,n))) &\leq&  \sum_{n= 0}^m \ell \left(\frac{\widetilde{{\mathcal F}^\prime_{n-2}}}{{\mathcal F}^\prime_{n-2}}\right )\\
&=&   \sum_{n= 0}^m \ell( R^\prime/{\mathcal F}^\prime_{n-2})-\sum_{n=0}^m\ell( R^\prime/\widetilde{{\mathcal F}^\prime_{n-2}})\\
&=& \widetilde{e}_3({\mathcal F}^\prime) -e_3({\mathcal F}^\prime)\\
&=&   \widetilde{e}_3({\mathcal F}^\prime) -e_3(\mathcal F) ~~\mbox{ (by \cite[Proposition 1.5]{gr})}\\
&=& \widetilde{e}_3({\mathcal F}^\prime) -e_3(I)
\end{eqnarray*}
This gives
\begin{equation}\label{e3}
0\leq \sum_{n\geq 0}\ell(H_2(C_.({\mathcal F}^\prime,n)))\leq \widetilde{e}_3( {\mathcal F}^\prime)-e_3(I).
\end{equation}
From (\ref{general-stat2}) and (\ref{e3}), we get
\begin{equation}\label{3}
r_J(I)\leq e_1(I)-e_0(I)+\ell(R/I)+1+\widetilde{e}_3({\mathcal F}^\prime)-e_3(I).
\end{equation}
By \cite[Proposition 4.4]{bl},  we have
for all $n\geq -1$,
\begin{equation} \label{diff}
P_{\widetilde{{\mathcal F}^\prime}}(n)-H_{\widetilde{{{\mathcal F}^\prime}}} (n)=\ell((H^2_{\mathcal R_+}(\mathcal R(\widetilde{{ {\mathcal F}^\prime}})))_{n+1}). 
\end{equation}
Now taking sum for $m\gg 0$ on both sides of the above equation, we get
\begin{eqnarray*}
\sum_{n=0}^m \ell((H^2_{\mathcal R_+}(\mathcal R(\widetilde{{ {\mathcal F}^\prime}})))_{n+1})&=& \sum_{n=0}^m P_{\widetilde{{\mathcal F}^\prime}}(n)-\sum_{n=0}^m H_{\widetilde{{{\mathcal F}^\prime}}} (n)\\
&=& \sum_{n=0}^m P_{\widetilde{{\mathcal F}^\prime}}(n)- H^2_{\widetilde{{{\mathcal F}^\prime}}} (m)
\\
&=& \widetilde{e}_0({{\mathcal F}^\prime}){m+3\choose 3}-\widetilde{e}_1({{\mathcal F}^\prime}){m+2\choose 2}+\widetilde{e}_2({{\mathcal F}^\prime}){m+1\choose 1}- P^2_{\widetilde{{{\mathcal F}^\prime}}} (m)\\
&=&\widetilde{e}_3({\mathcal F}^\prime).
\end{eqnarray*}
Since $R^\prime$ is a $2$-dimensional Cohen-Macaulay local ring,  we have $\ell((H^2_{\mathcal R_+}(\mathcal R(\widetilde{{ {\mathcal F}^\prime}})))_n)\leq \ell((H^2_{\mathcal R_+}(\mathcal R(\widetilde{{ {\mathcal F}^\prime}})))_{n-1})$ for all $n\in \mathbb Z$ by \cite[Lemma 4.7]{bl}. Now in equation \eqref{diff}, we substitute $n=-1$ to get $$\ell((H^2_{\mathcal R_+}(\mathcal R(\widetilde{{ {\mathcal F}^\prime}})))_0)=\widetilde{e}_2({ {\mathcal F}^\prime})=e_2({\mathcal F}^\prime)=e_2(\mathcal F)=e_2(I)$$
Therefore, 
\begin{equation}
\widetilde{e}_3( {\mathcal F}^\prime)=\sum_{n=0}^m \ell((H^2_{\mathcal R_+}(\mathcal R(\widetilde{{ {\mathcal F}^\prime}})))_{n+1})\leq\sum_{n=0}^ {a_2(\mathcal R(\widetilde{ {\mathcal F}^\prime}))}   \ell((H^2_{\mathcal R_+}(\mathcal R(\widetilde{{ {\mathcal F}^\prime}})))_0)= a_2(\mathcal R(\widetilde{ {\mathcal F}^\prime}))e_2(I).\nonumber
\end{equation}
where $a_2(\mathcal R(\widetilde{ {\mathcal F}^\prime}))\leq a_2(G(\widetilde{ {\mathcal F}^\prime})).$  By \cite[Corollary 5.7(2)]{tm}, $a_2(G(\widetilde{{\mathcal F}^\prime}))=\widetilde{r}({\mathcal F}^\prime)-2$.
Hence using \eqref{3}, we get
\begin{equation}\label{r_J(I3)}
    r_J(I)\leq e_1(I)-e_0(I)+\ell(R/I)+1+ (\widetilde{r}({\mathcal F}^\prime)-2)e_2(I)-e_3(I).\nonumber
\end{equation}
 By using Lemma \ref{4.6}, we get 
 $$
  r_J(I)\leq e_1(I)-e_0(I)+\ell (R/I)+1+e_2(I)(e_2(I)-e_1(I)+e_0(I)-\ell (R/I))-e_3(I).
 $$ 
 For $d\geq 4$, let $x\in I$ be a superficial element for $I$ such that $I^\prime = I/(x)$ is integrally closed. Then $e_i(I^\prime)=e_i(I)$ for $0\leq i \leq 3$ and $\depth G(I^\prime)\geq d-4.$ Since $\depth G(I)\geq 1$, we have $\widetilde{I^n}=I^n$ for all $n\geq 1$ which implies  $r_J(I)=r_{J^\prime}(I^\prime)$ by \cite[Lemma 3.1]{ms}.  
 Therefore, by induction, 
 \begin{align*}
 r_J(I)=r_{J^\prime}(I^\prime)&\leq e_1(I^\prime)-e_0(I^\prime)+\ell (R^\prime/I^\prime)+1+e_2(I^\prime)(e_2(I^\prime)-e_1(I^\prime)+e_0(I^\prime)-\ell (R^\prime/I^\prime))-e_3(I^\prime) \\ 
 &=e_1(I)-e_0(I)+\ell (R/I)+1+e_2(I)(e_2(I)-e_1(I)+e_0(I)-\ell (R/I))-e_3(I).\qedhere
 \end{align*}
\end{proof}
 When $\depth G(I)\geq d-3$, then by \cite[Theorem 4.1]{ms}, 
\begin{equation}\label{ms-bd}r_J(I)\leq e_1(I)-e_0(I)+\ell(R/I)+1+e_2(I)(e_2(I)-1)-e_3(I)\end{equation}
for an $\m$-primary ideal $I$. For an integrally closed ideal $I$, if $G(I)$ is not Cohen-Macaulay then $e_1(I)-e_0(I)+\ell(R/I)> 1$ which implies
 $e_2(I)-e_1(I)+e_0(I)-\ell(R/I)< e_2(I)-1.$ This shows that 
our bound  in Theorem \ref{improvment ms result} is an improvement of the above bound in \eqref{ms-bd}  when $I$ is integrally closed. The next example supports this. 
\begin{example}
In Example \ref{eg4}, let $d=3$. Then $r_Q(\m)=3$ and our bound $e_1(\m)-e_0(\m)+\ell (R/\m)+1+e_2(\m)(e_2(\m)-e_1(\m)+e_0(\m)-\ell (R/\m))-e_3(\m)=5$, where as the bound given in \cite[Theorem 4.1]{ms} is  $e_1(\m)-e_0(\m)+\ell(R/\m)+1+e_2(\m)(e_2(\m)-1)-e_3(\m)=17$.
\end{example}
We state below a case for integrally closed ideals when Rossi's bound holds.
\begin{corollary}
Let $(R, \m)$ be a Cohen-Macaulay local ring of dimension three and $I$ an $\m$-primary integrally closed ideal with $e_2(I)=e_1(I)-e_0(I)+\ell(R/I)$ and $e_3(I)\geq 0.$  Then 
$r_J(I)\leq e_1(I)-e_0(I)+\ell (R/I)+1$
for a minimal reduction $J\subseteq I$. 
\end{corollary}
{\bf{IV.}} Finally we remark that, using \eqref{eq-ei} and \eqref{eq-delta}, we have
\begin{equation}
{e_3}(\mathcal{I})=\sum\limits_{n\geq3}{{n-1\choose2 }\ell ({{I_n}}/{J{I_{n - 1}}})}-\sum\limits_{i=2}^d {{( - 1)}^i} \sum_{n\geq 3} {n-1 \choose 2} \ell (H_i
 (C_.(\mathcal{I},n))).\nonumber
\end{equation}
Suppose $d=3$. Then 
\begin{align*}
e_3(\mathcal{I})&=\sum\limits_{n\geq3}{n-1\choose 2} \ell (I_n/J{I_{n-1}})-\sum\limits_{n\geq 3}{n-1 \choose 2}\ell(H_2(C_.(\mathcal{I},n)))+\sum\limits_{n\geq 3}{n-1 \choose 2}\ell(H_3(C_.(\mathcal{I},n)))\\
&\leq \sum\limits_{n\geq3}{n-1 \choose 2}\ell (I_n/JI_{n-1})+\sum\limits_{n\geq 3}{n-1 \choose 2}\ell (H_3(C_.(\mathcal{I},n))).
  \end{align*}
From the modified Koszul complex $C_.(\mathcal{I},n)$, we have $H_3(C_.(\mathcal{I},n))=(I_{n-2}:J)/I_{n-3}$. Therefore, 
\begin{align*}
   e_3(\mathcal{I})\leq \sum\limits_{n\geq3}{n-1 \choose 2}[\ell (I_n/JI_{n-1})+\ell ((I_{n-2}:J)/I_{n-3})].  
\end{align*}
We obtain the following result as a consequence of Theorem \ref{T_{11}}  when $r_J(I)\leq 3.$ Recall that for an integrally closed ideal $I$, if $r_J(I)\leq 2$ then $G(I)$ is Cohen-Macaulay and $e_3(I)=0$ (see \cite[Theorem 3.6]{cpr}). 
\begin{theorem}\label{2.11}
Let $(R,\m)$ be a Cohen-Macaulay local ring of dimension $d\geq 3$, $I$ an $\m$-primary integrally closed ideal and $J$ be a minimal reduction of $I$. If $r_J(I)\leq 3$, then
\begin{align*}
    e_3({I})\leq \ell(I^3/JI^2)+\ell((J\cap I^3)/JI^2).
\end{align*}
\end{theorem}
\begin{proof}
 We have  $e_1(I)\geq \sum\limits_{n\geq 0}\ell(I^{n+1}/(J\cap I^{n+1}))=e_0(I)-\ell(R/I)+\sum\limits_{n\geq 1}\ell(I^{n+1}/(J\cap I^{n+1}))$ by \cite[Theorem 4.7(a)]{hm} and  $e_2(I)\leq \sum\limits_{n\geq 1}n\ell (I^{n+1}/JI^n)$ by \cite[Theorem 3.1]{cpr}. This gives, 
   $$e_2(I)-e_1(I)+e_0(I)-\ell(R/I)\leq \sum\limits_{n\geq 1}n\ell (I^{n+1}/JI^n)-\sum\limits_{n\geq  1}\ell(I^{n+1}/(J\cap I^{n+1})).$$
  Since $I$ is integrally closed, we have   $(I^2\cap J)=JI$. Now by Theorem \ref{T_{11}}, we get
 \begin{align*}
e_3(I)&\leq e_2(I)-e_1(I)+e_0(I)-\ell (R/I)\\
&\leq \ell(I^2/JI)+2\ell(I^3/JI^2)-\ell(I^2/(J\cap I^2))-\ell(I^3/(J\cap I^3))\\
&=\ell(I^3/JI^2)+\ell((J\cap I^3)/JI^2).\qedhere
 \end{align*}  
\end{proof}

\end{document}